\documentclass[a4paper,11pt]{amsart}
\usepackage{amssymb,amsthm,amsmath,amscd,graphicx,enumerate,color,enumerate}

\newtheorem{theorem}{Theorem}
\newtheorem*{theorem*}{Theorem}

\newtheorem{lemma}{Lemma}

\theoremstyle{definition}

\newtheorem{remark}{Remark}

\newcommand{\defn}{\ensuremath{\overset{\mathrm{def}}{=}}}
\newcommand{\dd}{\ensuremath{\mathrm{d}}}
\newcommand{\ii}{\ensuremath{\mathrm{i}}}

\renewcommand{\Im}{\ensuremath{\mathsf{Im}\,}}
\renewcommand{\Re}{\ensuremath{\mathsf{Re}\,}}
\renewcommand{\l}{\ensuremath{\left\langle}}
\renewcommand{\r}{\ensuremath{\right\rangle}}

\title[Distributional Boundary Data]{Elliptic boundary value problems in convex polygons with low regularity boundary data via the unified method}

\begin{document}
\author[A.C.L. Ashton]{A.C.L. Ashton}
       \address{DAMTP, University of Cambridge.}
       \email{a.c.l.ashton@damtp.cam.ac.uk}
\author[A.S. Fokas]{A.S. Fokas}
       \address{DAMTP, University of Cambridge.}
       \email{t.fokas@damtp.cam.ac.uk}

\date\today

\thanks{The first author is grateful for the support of Emmanuel College, University of Cambridge.}

\begin{abstract} 
We use novel integral representations developed by the second author to prove certain rigorous results concerning elliptic boundary value problems in convex polygons. Central to this approach is the so-called global relation, which is a non-local equation in the Fourier space that relates the known boundary data to the unknown boundary values. Assuming that the global relation is satisfied in the weakest possible sense, i.e. in a distributional sense, we prove there exist solutions to Dirichlet, Neumann and Robin boundary value problems with distributional boundary data. We also show that the analysis of the global relation characterises in a straightforward manner the possible existence of both integrable and non-integrable corner-singularities.
\end{abstract}

\maketitle

\section{Introduction}
In this paper we present a new means to study rigorous aspects of boundary value problems for elliptic PDEs in convex polygons. We concentrate on the basic elliptic equation
\begin{equation} -\Delta q + \beta^2 q =0, \label{basiceqn} \end{equation}
where $\beta^2>0$ (modified Helmholtz). The case $\beta=0$ (Laplace) was studied in \cite{fulton2004analytical} for continuous boundary data and can be extended to distributional boundary data using the methods presented here. All our results hold if $\beta^2\leq 0$ (Helmholtz or Laplace), but for economy of presentation we state our results for $\beta^2>0$ only.

We study (a) existence of solutions with distributional boundary values, (b) a priori estimates for boundary data regularity and (c) corner singularities. The approach we use is based on the formal results obtained in \cite{fokas1997unified} where novel integral representations for solutions to the basic elliptic equations in a convex polygon were derived, under the \emph{assumption} that the solution exists. The rigorous methodology is based on the following procedure: we define a function using the novel integral representations of \cite{fokas1997unified} and show that this function satisfies the relevant PDE and converges to the desired data on the boundary provided that the global relation is satisfied.

Regarding (a), we first generalise the integral representations in \cite{fokas1997unified} so they make sense for distributional boundary values. We then show that the function defined by the associated integral representation satisfies the boundary value problem with distributional boundary values \emph{provided} that the unknown boundary values satisfy the global relation in the distributional sense.

Regarding (b), we provide classical a priori estimates for the unknown boundary data by analysing the global relation.

Regarding (c), we use the global relation to prove the existence of corner singularities when certain mixed boundary data are prescribed at adjacent edges of the polygon. The arguments are based on simple asymptotic evaluation of terms in the global relation.

In this paper we do not address the fundamental question of existence of solution to the global relation. It is known that if Dirichlet data is in $H^{s+1/2}(\partial\Omega)$ with $s\in [-\tfrac{1}{2},\tfrac{1}{2}]$, then the global relation uniquely defines the unknown boundary values, which belongs to $H^{s-1/2}(\partial\Omega)$ \cite{ashton2012rigorous}. A new, more direct method for solving the Dirichlet-Neumann map via the global relation was presented in \cite{ashton2012spectral} for Dirichlet Data in $H^1(\partial\Omega)$. This method can be extended to boundary data with lower regularity. We expect that if one works on a suitable quotient space, then regularity results for the Dirichlet-Neumann map will hold for general distributional boundary data. The relevant space quotients out finite linear combinations of Dirac measures at the vertices (the relevance of this space is discussed in Theorem \ref{unique}).

For a more classical approach to boundary value problems with distributional boundary data, we refer the reader to \cite{babuvska2008boundary,dauge1988elliptic} and references therein.

Several numerical techniques for the implementation of the unified method of \cite{fokas1997unified,fokas2000integrability} applied to linear elliptic PDEs are discussed in \cite{fornberg2011numerical,fulton2004analytical,sifalakis2008generalized,smitheman2010spectral}.

\section{The Spaces of distributions}
Let $\Omega \subset \mathbf{R}^2 \simeq \mathbf{C}$ denote the interior of a polygon with vertices $\{z_i\}_{i=1}^n$ and sides $\Gamma_i = (z_i,z_{i+1})$. Let $\alpha_i = \arg(z_{i+1}-z_i)$ denote the angle the side $\Gamma_i$ makes with the positive real axis and let $|\Gamma_i|$ be the length of the side $\Gamma_i$. It was shown in \cite{fokas1997unified} that if one assumes that there exists a solution to \eqref{basiceqn} in $\Omega$ with sufficient smoothness, then the solution has the following representation:
\begin{equation} q(z,\bar{z}) = \frac{1}{4\pi\ii} \sum_{i=1}^n \int_{\ell_i} e^{\ii \lambda z - \ii \beta^2 \bar{z}/\lambda} \rho_i(\lambda)\, \frac{\dd\lambda}{\lambda}, \label{rep} \end{equation}
where the functions $\{\rho_i(\lambda)\}_{i=1}^n$, called spectral functions, are defined by 
\begin{equation} \rho_i(\lambda) = \int_{\Gamma_i} e^{-\ii\lambda z'+\ii\beta^2 \bar{z}'/\lambda}\left[ \left( \frac{\partial q}{\partial z'}+\ii \lambda q\right) \dd z' + \left(\frac{\ii\beta^2}{\lambda} q-\frac{\partial q}{\partial \bar{z}'} \right) \dd\bar{z}'\right] \end{equation}
and  $\{\ell_i\}_{i=1}^n$ denote the rays in the complex plane with $\arg (\lambda|_{\ell_i})= -\alpha_i$, orientated towards infinity. Furthermore, the spectral functions satisfy the global relation
\begin{equation} \sum_{i=1}^n \rho_i(\lambda) =0. \label{gr1} \end{equation}
In what follows we will first generalise these results so that they make sense for distributions (generalised functions). In this respect we will use standard results from the theory of distributions, a comprehensive account of which can be found in \cite{hormander1985analysis1}. We will be interested in distributions supported on the closure of the edges $\Gamma_i$ (the edges $\Gamma_i$ do not contain the vertices). We use $\mathcal{E}[0,1]$ to denote the space of smooth functions from $[0,1]$ to $\mathbf{C}$. This space is endowed with a topology generated from the semi-norms
\[ \delta_n(\varphi) = \max_{\tau \in [0,1]} \left| \varphi^{(n)}(\tau)\right|, \quad n=0,1,2,\ldots , \]
a sequence of functions $\varphi_k\rightarrow 0$ in $\mathcal{E}[0,1]$ if and only if $\delta_n(\varphi_k)\rightarrow 0$ for all $n\in \mathbf{N}\cup \{0\}$. This is a Frech\'{e}t space and we denote its topological dual by $\mathcal{E}'[0,1]$, which is the space of continuous linear forms from $\mathcal{E}[0,1]$ to $\mathbf{C}$. The pairing between $\mathcal{E}'[0,1]$ and $\mathcal{E}[0,1]$ is denoted by $\l \cdot, \cdot \r$. It is well known that a linear form $u:\mathcal{E}[0,1]\rightarrow \mathbf{C}$ is an element of $\mathcal{E}'[0,1]$ if and only if there are constants $N$ and $C$ such that
\begin{equation} \left| \l u, \varphi \r \right| \leq C \sum_{n\leq N}\delta_n(\varphi), \label{seminorms} \end{equation}
for all $\varphi \in \mathcal{E}[0,1]$.

We will also make use of the less familiar space of Schwartz distributions on an open interval. To define this space of distributions, we first setup the appropriate space of test functions. Let $\mathcal{S}(0,1)$ denote the space of smooth functions from $(0,1)$ to $\mathbf{C}$ such that
\[ \sigma_{mn}(\varphi) = \sup_{\tau \in (0,1)} | \tau^{-m} (1-\tau)^{-m} \varphi^{(n)}(\tau) | < \infty,\quad n,m = 0,1,2,\ldots ,\]
so that these functions vanish to all orders at the boundary of $[0,1]$. The topology generated by these semi-norms makes $\mathcal{S}(0,1)$ a Frech\'{e}t space. The topological dual, consisting of continuous linear forms from $\mathcal{S}(0,1)$ to $\mathbf{C}$, is denoted by $\mathcal{S}'(0,1)$. A linear form $v:\mathcal{S}(0,1)\rightarrow \mathbf{C}$ is an element of $\mathcal{S}'(0,1)$ if and only if there exist constants $D$ and $M$ such that
\[ \left| \l v, \varphi \r \right| \leq D \sum_{m,n\leq M} \sigma_{mn}(\varphi), \]
for all $\varphi \in \mathcal{S}(0,1)$. For $\varphi \in \mathcal{S}(0,1)$ we define the Fourier transform by
\[ \hat{\varphi}(\lambda) = \int_0^1 e^{-\ii \lambda \tau} \varphi (\tau)\, \dd \tau. \]
If $\check{\varphi}(x) = \varphi(-x)$, the Fourier inversion theorem can be written as $2\pi \varphi = ((\hat{\varphi})\hat{\,})\check{\,}$. Using the natural extension of an element of $\mathcal{E}'[0,1]$ to $\mathcal{E}'(\mathbf{R})$, the space of distributions on $\mathbf{R}$ with compact support, we can define the Fourier transform on $\mathcal{E}'[0,1]$ via
\[ \hat{u}(\lambda) = \l u_\tau, e^{-\ii\lambda \tau} \r. \]
See \cite[Ch. 7]{hormander1985analysis1} for the standard treatment of the Fourier transform on $\mathcal{E}'(\mathbf{R})$.

We now generalise the integral representation \eqref{rep} so it makes sense for distributional boundary values. To facilitate this generalisation we need a local description of the edges $\Gamma_i$ so that we can define distributions on them. We use the local parametrisations $\psi_i:[0,1]\rightarrow \Gamma_i$, with
\[ \psi_i: \tau\mapsto \tau z_{i+1} + (1-\tau)z_i. \]
For a function $f:\Gamma_i\rightarrow \mathbf{C}$ we write its pullback by $\psi_i$ via
\[ \psi_i^*(f)(\tau) = f(\psi(\tau)). \]
1-forms are pulled back by $\psi_i$ via
\[ \psi_i^*(\dd z) = e^{\ii \alpha_i} |\Gamma_i| \, \dd \tau, \quad \psi_i^*(\dd \bar{z}) = e^{-\ii \alpha_i} |\Gamma_i|\, \dd \tau. \]
We now define the spectral functions $\rho_i(\lambda)$ by
\begin{align}
\frac{\rho_i(\lambda)}{\ii |\Gamma_i|} &= \left\langle \partial_\mathbf{n} q_i, \psi_i^*\!\left(e^{-\ii\lambda z' + \ii \beta^2 \bar{z}'/\lambda}\right)\right\rangle \label{rho1}\\ 
& \quad + \left(\lambda e^{\ii\alpha_i} + \frac{\beta^2}{\lambda e^{\ii\alpha_i} }\right) \left \langle  q_i, \psi_i^*\!\left(e^{-\ii\lambda z' + \ii \beta^2 \bar{z}'/\lambda}\right)\right\rangle, \nonumber  \end{align}
where $\partial_\mathbf{n} q_i, q_i \in \mathcal{E}'[0,1]$ for $i=1,\ldots,n$. The first term is motivated by the identification
\begin{align*} \int_{\Gamma_i} e^{-\ii \lambda z'+\ii \beta^2 \bar{z}'/\lambda}\left( \frac{\partial q}{\partial z'}\, \dd z' - \frac{\partial q}{\partial\bar{z}'}\, \dd \bar{z}'\right) 
&=  |\Gamma_i| \int_0^1 \psi_i^*\!\left[ e^{-\ii \lambda z'+\ii \beta^2 \bar{z}'/\lambda} \frac{\partial q}{\partial \mathbf{n}} \right]\!(\tau) \,\dd \tau \\
& \rightsquigarrow |\Gamma_i|\l \partial_\mathbf{n} q_i, \psi_i^*\!\left(e^{-\ii\lambda z' + \ii \beta^2 \bar{z}'/\lambda}\right)  \r.
\end{align*}
In this calculation we have used the identities
\[ \frac{\partial q}{\partial z'}\bigg|_{\Gamma_i} = \frac{e^{-\ii \alpha_i}}{2} \left( \frac{\partial q}{\partial \mathbf{t}}+\ii \frac{\partial q}{\partial \mathbf{n}}\right)\bigg|_{\Gamma_i}, \qquad  \frac{\partial q}{\partial \bar{z}'}\bigg|_{\Gamma_i} = \frac{e^{\ii \alpha_i}}{2} \left( \frac{\partial q}{\partial \mathbf{t}}-\ii \frac{\partial q}{\partial \mathbf{n}}\right)\bigg|_{\Gamma_i}, \]
where $\partial/\partial \mathbf{t}$ and $\partial/\partial \mathbf{n}$ are derivatives in the tangential and (outward) normal directions along $\Gamma_i$. The second term in \eqref{rho1} is similarly motivated, e.g.
\begin{align*} \frac{\ii \beta^2}{\lambda} \int_{\Gamma_i} e^{-\ii \lambda z'+\ii \beta^2 \bar{z}'/\lambda} q\, \dd \bar{z}' &= \frac{\ii \beta^2|\Gamma_i|e^{-\ii \alpha_i}}{\lambda} \int_0^1 \psi_i^*\!\left[ e^{-\ii \lambda z'+\ii \beta^2 \bar{z}'/\lambda} q\right]\! (\tau)\, \dd \tau \\
&\rightsquigarrow\frac{\ii \beta^2|\Gamma_i|e^{-\ii\alpha_i}}{\lambda}\left \langle  q_i, \psi_i^*\!\left(e^{-\ii\lambda z' + \ii \beta^2 \bar{z}'/\lambda}\right)\right\rangle.
\end{align*}
So the formal identifications are
\[ \partial_\mathbf{n} q_i \rightsquigarrow \psi_i^*\!\left( \frac{\partial q}{\partial \mathbf{n}}\right), \quad q_i \rightsquigarrow \psi_i^*(q). \]

\section{Distributional Boundary Data}
In this section we will outline how to setup a given boundary value problem so that the boundary data can be distributional. Our aim is to show that the function defined by \eqref{rep}, where the $\{\rho_i\}_{i=1}^n$ are given in \eqref{rho1}, satisfies the following boundary value problem:
\begin{subequations}\label{bvp}
\begin{align} -\Delta q + \beta^2 q &=0 \,\quad \textrm{in $\Omega$} \\
\lim_{z\rightarrow \Gamma_i} q &= q_i \quad \textrm{in $\mathcal{S}'(\Gamma_i)$ for $i=1,\ldots,n$,}
\end{align}
\end{subequations}
where $q_i\in \mathcal{E}'[0,1]$. We need to specify a notion of convergence, via which the solution \eqref{rep} converges to a distribution on $\Gamma_i$. 

For a given polygon $\Omega$ we inscribe within it a one parameter family of polygons
\[ \Omega_\epsilon = \{ z\in \Omega: \mathrm{dist}(z,\partial\Omega)\geq \epsilon\}. \]
For $\epsilon>0$ sufficiently small we have $\partial\Omega \simeq \partial\Omega_\epsilon$ and we can define the edges of $\partial\Omega_\epsilon$ by $\Gamma_i^\epsilon$ for $i=1,\ldots,n$ as shown in Figure \ref{epsdomain}. Now suppose $q\in C^\infty(\Omega)$ is a given function. Using a local parametrisation $\psi_i^\epsilon:[0,1] \rightarrow \Gamma_i^\epsilon$, we can define a distribution $q_i^\epsilon \in \mathcal{S}'(0,1)$ by
\begin{equation} \l q_i^\epsilon, \varphi \r = \int_0^1 \left( \psi_i^\epsilon \right)^*\!\!(q)(\tau) \varphi(\tau)\, \dd \tau, \label{distlim} \end{equation}
for all $\varphi \in \mathcal{S}(0,1)$. As $\epsilon \downarrow 0$ we have $\Gamma_i^\epsilon \rightarrow \Gamma_i$ and $\psi_i^\epsilon \rightarrow \psi_i$. This limit and its equivalence to others modes of convergence for boundary values are discussed in \cite{babuvska2008boundary} (cf. \cite{seeley1966singular}). Our definition of $q\in C^\infty(\Omega)$ tending to a distribution on $\Gamma_i$ is
\[ \Big\{ \lim_{z\rightarrow \Gamma_i} q = q_i \,\,\, \textrm{in $\mathcal{S}'(\Gamma_i)$} \Big\} \quad \Leftrightarrow \quad \Big\{ \lim_{\epsilon\rightarrow 0} q_i^\epsilon = q_i \,\,\, \textrm{in $\mathcal{S}'(0,1)$} \Big\}. \]
\begin{figure}\label{epsdomain}
\begin{center}
\includegraphics[scale=0.8]{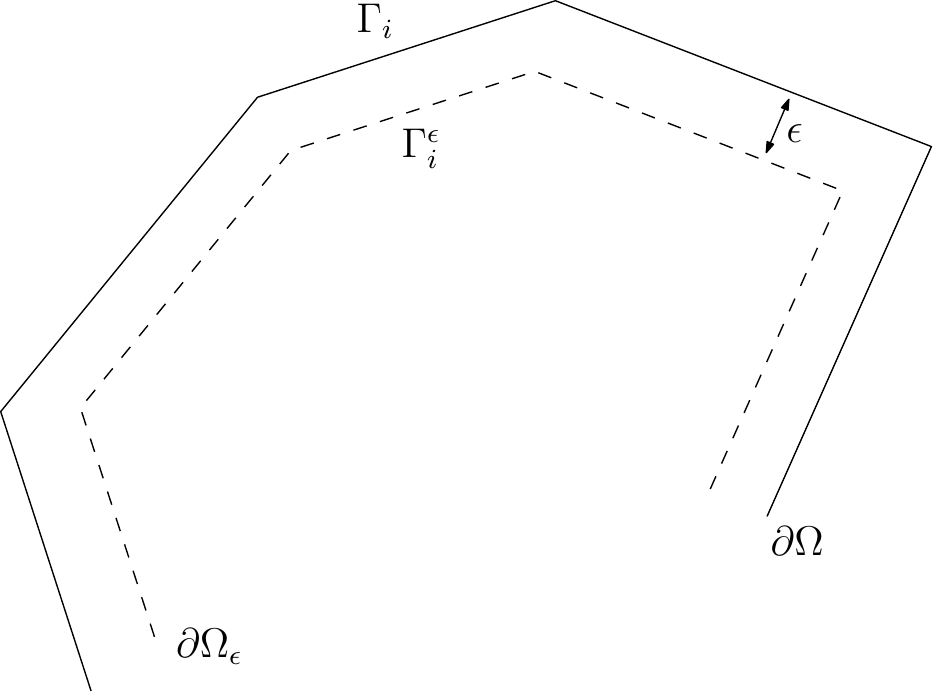}
\end{center}
\caption{\small The boundaries $\partial\Omega$ and $\partial\Omega_\epsilon$.}
\end{figure}
These limits are understood in the topology of $\mathcal{S}'(0,1)$. Recall that our boundary datum is an element of $\mathcal{E}'[0,1]$; on the other hand we are taking limits in the topology of $\mathcal{S}'(0,1)$ and \emph{not} in the topology of $\mathcal{E}'[0,1]$ (the test function appearing in \eqref{distlim} is an element of $\mathcal{S}(0,1)$ and \emph{not} of $\mathcal{E}[0,1]$). Thus, we must make sure that we do not lose any information in this limit process.

In this respect we note $\mathcal{S}(0,1) \subset \mathcal{E}[0,1]$ and that the inclusion is continuous\footnote{Meaning that if a sequence $\{\varphi_k\}_{k\geq 1}$ tends to zero in $\mathcal{S}(0,1)$ (i.e. $\sigma_{mn}(\varphi_k)\rightarrow 0$ for all $m,n$), it also tends to zero in $\mathcal{E}[0,1]$ (i.e. $\delta_n(\varphi_k)\rightarrow 0$ for all $n$).}. Hence $\mathcal{E}'[0,1] \hookrightarrow \mathcal{S}'(0,1)$. Therefore there is a natural projection $\pi : \mathcal{E}'[0,1] \rightarrow \mathcal{S}'(0,1)$, which is obtained by restricting $u\in \mathcal{E}'[0,1]$ to test functions in $\mathcal{S}(0,1)$. Given a linear form $v\in \mathcal{S}'(0,1)$, the Hahn-Banach theorem implies that we can extend this linear form to a linear form $u_0\in \mathcal{E}'[0,1]$ with $\pi (u_0)=v$. The general solution to the equation $\pi (u)=v$ is \cite[p. 178]{estrada2000singular}
\[ u = u_0 + \sum_{k=0}^N a_k \delta^{(k)}_{0} + \sum_{k=0}^N a_k \delta^{(k)}_{1}. \]
This observation means that the difference between $u\in \mathcal{E}'[0,1]$ and $\pi(u)$ can only be a collection of Dirac measures and derivatives thereof supported at the boundary of $[0,1]$. Intuitively speaking this is because the test functions in $\mathcal{S}(0,1)$ vanish to all orders at the boundary of $[0,1]$, so they are unable to ``detect" a distribution in $\mathcal{E}'[0,1]$ with support contained in the boundary of $[0,1]$. These observations imply that knowledge of
\[  \lim_{z\rightarrow \Gamma_i} q = q_i \,\,\, \textrm{in $\mathcal{S}'(\Gamma_i)$}  \]
yields
\begin{equation} \lim_{z\rightarrow \Gamma_i} q= q_i + \sum_{j=0}^N \left( a_j \delta^{(j)}_{z_i} + b_j \delta^{(j)}_{z_{i+1}}\right) \,\,\, \textrm{in $\mathcal{E}'(\Gamma_i)$}, \label{diraccontrib} \end{equation}
for some constants $\{a_j,b_j \}_{j=0}^N$ and some integer $N$. The possibility of the existence of the above additional contributions to the right hand side of \eqref{diraccontrib} does not affect the answer to the questions of existence of a solution to \eqref{bvp}, since we can always consider the new function (distribution)
\[ \tilde{q} = q - \sum_{i=1}^n \sum_{j=0}^N c_{ij} \delta^{(j)}_{z_i}, \]
for appropriate $\{ c_{ij}\}$ which satisfies the relevant PDE in $\Omega$ and converges to the desired distribution in $\mathcal{E}'[0,1]$.

The following result shows that $\mathcal{S}'(0,1)$ is, in some sense, the right space of distributions to take limits in.
\begin{theorem}\label{unique}
Let $\Delta[0,1]$ denote the space of all finite linear combinations of Dirac measures and derivatives thereof supported at $0$ and $1$. Given $q_i \in \mathcal{E}'[0,1]$ ($i=1,\ldots ,n)$, there is at most one of each $\partial_\mathbf{n}q_i \in \mathcal{E}'[0,1]/\Delta[0,1]$ ($i=1,\ldots, n$) such that the global relation is satisfied.
\end{theorem}
This result implies that a solution of the global relation is unique modulo Dirac measures at the vertices of the polygon. By taking limits in $\mathcal{S}'(0,1)$ we effectively quotient out this degeneracy. In this sense, the topology on $\mathcal{S}'(0,1)$ is the right one to take limits in for distributional boundary data. A more detailed treatment of distributional solutions to the global relation will be presented elsewhere, but the interested reader may look at \cite[Lemma 4]{ashton2012spectral} for an idea of the proof.

To prove our main result we will need to interchange the order of some integrals. To this end, the following simple lemma is essential.
\begin{lemma}\label{expconverge}
For $z\in \Omega$ and for each $i$, the function
\[ \frac{1}{\lambda}e^{\ii \lambda z - \ii \beta^2\bar{z}/\lambda} \rho_i(\lambda) \]
decays exponentially as $|\lambda|\rightarrow 0$ and $|\lambda|\rightarrow \infty$ along the ray $\ell_i$.
\end{lemma}
\begin{proof}
From the semi-norm estimates \eqref{seminorms} we know that there exist constants $C,N$ such that
\[ |\rho_i(\lambda)| \leq C \sum_{n=0}^N \left( |\lambda| +\frac{\beta^2}{|\lambda|}\right)^n \max_{\mathrm{cl}(\Gamma_i)} \left| e^{-\ii \lambda z' + \ii \beta^2 \bar{z}'/\lambda}\right|. \]
The proof now follows from calculations analogous to those in \cite[Lemma 1]{fulton2004analytical}.
\end{proof}
From this observation it is clear that the double integrals
\[ \int_0^1 \varphi(\tau) \left( \psi_i^\epsilon\right)^*\!\! \left[ \int_{\ell_j} e^{\ii \lambda z - \ii \beta^2 \bar{z}/\lambda} \rho_j(\lambda)\, \frac{\dd \lambda}{\lambda}\right]\!\!(\tau)\,\dd \tau \]
converge absolutely for each $1\leq i,j\leq n$ and $\epsilon>0$. Hence applying Fubini's theorem we can interchange the order of integration
\[ \int_{\ell_j} \rho_j(\lambda) \left[ \int_0^1 \left(\psi_i^\epsilon\right)^*\!\left( e^{\ii \lambda z - \ii \beta^2 \bar{z}/\lambda}\right)\! \varphi(\tau)\, \dd \tau \right] \frac{\dd \lambda}{\lambda}. \]
\begin{lemma}\label{intlem}
Let $\hat{\ell}_i$ denote the ray for which $\arg(\lambda|_{\hat{\ell}_i})=\pi - \alpha_i$, so that $\hat{\ell}_i$ is the continuation of the ray $\ell_i$. Then, for $j\neq i$:
\begin{align*} & \int_{\ell_j} \rho_j(\lambda)  \left[ \int_0^1 \psi_i^*\!\left(  e^{\ii \lambda z - \ii \beta^2 \bar{z}/\lambda}\right)\!\varphi(\tau)\,\dd \tau\right] \frac{\dd \lambda}{\lambda} \\ &\qquad = \int_{\hat{\ell}_i} \rho_j(\lambda) \left[ \int_0^1 \psi_i^*\!\left(  e^{\ii \lambda z - \ii \beta^2 \bar{z}/\lambda}\right)\!\varphi(\tau)\,\dd \tau\right] \frac{\dd \lambda}{\lambda}.
\end{align*}
\end{lemma}
\begin{proof}
In order to be able to deform contours in the complex plane, we need to establish some estimates for the integrand when $|\lambda|$ is either small or large (Cauchy's theorem will then provide the result since the integrand is analytic in any punctured disc centred at the origin). In this respect, for each $m\geq 0$ and some constant $C_m$ we have the following estimates:
\begin{align*} 
&\left| \int_0^1 \psi_i^*\!\left(  e^{\ii \lambda z - \ii \beta^2 \bar{z}/\lambda}\right)\!\varphi(\tau)\,\dd \tau\right| \\ &\qquad\leq C_m \left( |\lambda| + \frac{\beta^2}{|\lambda|}\right)^{-m} \left|\int_0^1 \psi_i^*\!\left(  e^{\ii \lambda z - \ii \beta^2 \bar{z}/\lambda}\right)\!\varphi^{(m)}(\tau)\,\dd \tau \right| \\
&\qquad\leq C_m \left( |\lambda| + \frac{\beta^2}{|\lambda|}\right)^{-m} \max_{\mathrm{cl}(\Gamma_i)} \left| e^{\ii \lambda z - \ii \beta^2 \bar{z}/\lambda}\right|.
\end{align*}
This estimate follows from repeated integration by parts, using the fact that $\varphi$ and its derivatives vanish to all orders at the boundary $\partial \Gamma_i$. The semi-norm estimates used in Lemma \ref{expconverge} imply that for $m$ and $C_m$ as above, we have
\begin{align*} & \left| \rho_j(\lambda) \int_0^1 \psi_i^*\!\left(  e^{\ii \lambda z - \ii \beta^2 \bar{z}/\lambda}\right)\!\varphi(\tau)\,\dd \tau\right| \\ &\qquad \leq C_m \left(|\lambda| + \frac{\beta^2}{|\lambda|}\right)^{-m} \max_{z\in \mathrm{cl}(\Gamma_i)} \max_{z'\in \mathrm{cl}(\Gamma_j)} \left| e^{\ii \lambda (z-z') - \ii \beta^2 (\bar{z}-\bar{z}')/\lambda} \right|.
\end{align*}
Observe that
\begin{equation} \left| e^{\ii \lambda (z-z') - \ii \beta^2 (\bar{z}-\bar{z}')/\lambda} \right| = \exp\left( -|z-z'| \left(|\lambda|+\frac{\beta^2}{|\lambda|}\right) \sin (\arg(\lambda(z-z')))\right). \label{expest}\end{equation}
If $|\alpha_i-\alpha_j|=\pi$, then $\ell_j$ and $\hat{\ell}_i$ coincide, thus there is nothing to prove. Let us suppose $0<\alpha_i - \alpha_j<\pi$. This condition together with convexity implies
\[ \alpha_i - \pi<\alpha_j < \arg(z-z') < \alpha_j + \pi.  \]
Hence, if $\lambda$ satisfies $\pi - \alpha_i \leq \arg \lambda \leq -\alpha_j$, we find
\[ 0 \leq \arg(\lambda(z-z'))\leq \pi. \]
It follows now from \eqref{expest} that the ray $\ell_j$ can be deformed onto $\hat{\ell}_i$. The case $0<\alpha_j-\alpha_i<\pi$ is similar.
\end{proof}

\begin{theorem}\label{thm1}
Define the function $Q=Q(z,\bar{z})$ by
\[ Q(z,\bar z) = \frac{1}{4\pi\ii} \sum_{i=1}^n \int_{\ell_i} e^{\ii \lambda z - \ii \beta^2 \bar{z}/\lambda} \rho_i(\lambda)\, \frac{\dd\lambda}{\lambda} \]
with $\{\rho_i(\lambda)\}_{i=1}^n$ defined in \eqref{rho1}, where $q_i,\partial_\mathbf{n} q_i \in \mathcal{E}'[0,1]$ for $i=1,\ldots, n$. If the global relation \eqref{gr1} is satisfied, then $Q$ solves the following boundary value problem:
\begin{subequations}\label{bvp'}
\begin{align} -\Delta Q + \beta^2 Q &=0 \,\quad \textrm{in $\Omega$} \\
\lim_{z\rightarrow \Gamma_i} Q &= q_i \quad \textrm{in $\mathcal{S}'(\Gamma_i)$ for $i=1,\ldots,n$.}
\end{align}
\end{subequations}
\end{theorem}
\begin{proof}
It is straightforward to show that $Q$ satisfies the relevant equation at any interior point of $\Omega$, since exponential convergence justifies differentiating under the integral sign. Thus, we focus on the behaviour of $Q$ at the boundary. Without loss of generality we position the vertex $z_{i}$ at the origin and align the side $\Gamma_i$ with the positive real $z$-axis so that $\alpha_i=0$. Our aim is to compute the limit
\[ \lim_{\epsilon\rightarrow 0} \l Q_i^\epsilon, \varphi \r = \lim_{\epsilon\rightarrow 0} \int_0^1 \left(\psi_i^\epsilon\right)^*\!(Q)(\tau) \varphi (\tau)\, \dd \tau. \]
We have already established that the double integral converges absolutely, so interchanging the order of integration is justified. Thus, our task is to compute the limit
\[ \lim_{\epsilon\rightarrow 0}\frac{1}{4\pi\ii} \sum_{j=1}^n \int_{\ell_j} \rho_j(\lambda) \left[ \int_0^1 \left(\psi_i^\epsilon\right)^*\!\left(  e^{\ii \lambda z - \ii \beta^2 \bar{z}/\lambda}\right)\!\varphi(\tau)\,\dd \tau \right]\, \frac{\dd \lambda}{\lambda}. \]
It is simple to bound each of the integrands by an absolutely integrable function (independent of $\epsilon$), employing the same integration by parts type estimates used in Lemma \ref{intlem}. So the dominated convergence theorem applies and we can pass the limit through the integral signs. Hence,
\begin{align*}
 \lim_{\epsilon\rightarrow 0} \l Q_i^\epsilon, \varphi \r &= \frac{1}{4\pi\ii} \sum_{j=1}^n \int_{\ell_j} \rho_j(\lambda) \left[ \int_0^1 \psi_i^*\!\left(  e^{\ii \lambda z - \ii \beta^2 \bar{z}/\lambda}\right)\!\varphi(\tau)\,\dd \tau \right] \frac{\dd \lambda}{\lambda}  \\
 &= \frac{1}{4\pi\ii}\int_{\ell_i} \rho_i(\lambda) \left[ \int_0^1 \psi_i^*\!\left(  e^{\ii \lambda z - \ii \beta^2 \bar{z}/\lambda}\right)\!\varphi(\tau)\,\dd \tau \right] \frac{\dd \lambda}{\lambda} \\
 &\quad\quad + \frac{1}{4\pi\ii} \sum_{j\neq i} \int_{\ell_j} \rho_j(\lambda) \left[ \int_0^1 \psi_i^*\!\left(  e^{\ii \lambda z - \ii \beta^2 \bar{z}/\lambda}\right)\!\varphi(\tau)\,\dd \tau \right] \frac{\dd \lambda}{\lambda}.
\end{align*}
Using Lemma \ref{intlem} and the global relation \eqref{gr1} we find
\begin{align*}
& \frac{1}{4\pi\ii} \sum_{j\neq i} \int_{\ell_j} \rho_j(\lambda) \left[ \int_0^1 \psi_i^*\!\left(  e^{\ii \lambda z - \ii \beta^2 \bar{z}/\lambda}\right)\!\varphi(\tau)\,\dd \tau \right] \frac{\dd \lambda}{\lambda} \\
&\qquad\qquad =  \frac{1}{4\pi\ii}  \int_{\hat{\ell}_i} \Bigg( \sum_{j\neq i} \rho_j(\lambda) \Bigg)\left[ \int_0^1 \psi_i^*\!\left(  e^{\ii \lambda z - \ii \beta^2 \bar{z}/\lambda}\right)\!\varphi(\tau)\,\dd \tau \right] \frac{\dd \lambda}{\lambda}\\
 &\qquad\qquad = - \frac{1}{4\pi\ii}\int_{\hat{\ell}_i} \rho_i(\lambda) \left[ \int_0^1 \psi_i^*\!\left(  e^{\ii \lambda z - \ii \beta^2 \bar{z}/\lambda}\right)\!\varphi(\tau)\,\dd \tau \right] \frac{\dd \lambda}{\lambda}.
\end{align*}
Hence
\begin{equation} \lim_{\epsilon\rightarrow 0}\l Q_i^\epsilon, \varphi \r = \frac{1}{4\pi \ii} \left( \int_{\ell_i}- \int_{\hat{\ell}_i}\right) \rho_i(\lambda) \left[ \int_0^1 \psi_i^*\!\left(  e^{\ii \lambda z - \ii \beta^2 \bar{z}/\lambda}\right)\!\varphi(\tau)\,\dd \tau\right]\, \frac{\dd \lambda}{\lambda}. \label{weakok} \end{equation}
Since $\alpha_i=0$, we can parametrise the $\lambda$-integrals in \eqref{weakok} so that $\ell_i$ is the positive real axis and $\hat{\ell_i}$ is the negative real axis:
\begin{align*}   &4\pi \ii\,\lim_{\epsilon\rightarrow 0} \l Q_i^\epsilon,\varphi \r  \\ &= \int_{0}^{\infty}  \hat{\varphi}\left(|\Gamma_i|\Big(\tfrac{\beta^2}{\lambda}-\lambda\Big)\right) \rho_i(\lambda)\, \frac{\dd\lambda}{\lambda} -  \int_0^{-\infty}  \hat{\varphi}\left(|\Gamma_i|\Big(\tfrac{\beta^2}{\lambda}-\lambda\Big)\right) \rho_i(\lambda)\, \frac{\dd\lambda}{\lambda} \\
&= \int_{0}^{\infty}  \hat{\varphi}\left(|\Gamma_i|\Big(\tfrac{\beta^2}{\lambda}-\lambda\Big)\right) \rho_i(\lambda)\, \frac{\dd\lambda}{\lambda} - \int_{0}^{\infty}  \hat{\varphi}\left(|\Gamma_i|\Big(\lambda -\tfrac{\beta^2}{\lambda}\Big)\right) \rho_i(-\lambda)\, \frac{\dd\lambda}{\lambda}.
\end{align*}
In the first integral we make the change of variables $k = |\Gamma_i|(\lambda-\beta^2/\lambda)$, i.e.
\[ 2\lambda =\frac{1}{|\Gamma_i|}\left( k+\sqrt{k^2+4|\Gamma_i|^2\beta^2}\right), \quad \frac{\dd \lambda}{\lambda} = \frac{\dd k}{\sqrt{k^2+4|\Gamma_i|^2\beta^2}} , \]
where we have chosen the positive square root so that $\lambda\geq 0$. Similarly, in the second integral we make the change of variables $k=-|\Gamma_i|(\lambda-\beta^2/\lambda)$, i.e.
\[ 2\lambda = \frac{1}{|\Gamma_i|}\left(-k+\sqrt{k^2+4|\Gamma_i|^2\beta^2}\right), \quad \frac{\dd\lambda}{\lambda} = -\frac{\dd k}{\sqrt{k^2+4|\Gamma_i|^2\beta^2}}. \]
Hence
\begin{align}  4\pi \ii\, \lim_{\epsilon\rightarrow 0} \l Q_i^\epsilon, \varphi \r &= \int_{-\infty}^{\infty}  \hat{\varphi}(-k)\left( \frac{ \ii |\Gamma_i| (\partial_\mathbf{n} q_i)\hat{\,}(k)}{\sqrt{k^2+4|\Gamma_i|^2\beta^2}} + \ii  \hat{q}_i(k)\right)\dd k\nonumber \\
 &\quad\quad -  \int_{-\infty}^{\infty}  \hat{\varphi}(-k)\left( \frac{ \ii |\Gamma_i| (\partial_\mathbf{n} q_i)\hat{\,}(k)}{\sqrt{k^2+4|\Gamma_i|^2\beta^2}} - \ii \hat{q}_i(k)\right)\dd k .\label{weak}
\end{align}
We note that the above integrals are well defined for $q_i, \partial_\mathbf{n} q_i \in \mathcal{E}'[0,1]$, since by the semi-norm estimates \eqref{seminorms}, the corresponding Fourier transforms are entire functions of polynomial growth on the real line, . We find the contributions from the derivatives cancel, and we are left with the expression
\begin{equation}  \lim_{\epsilon\rightarrow 0} \l Q_i^\epsilon, \varphi \r =  \frac{1}{2\pi} \int_{-\infty}^\infty  \hat{\varphi}(-k) \hat{q}_i(k)\, \dd k. \label{almost} \end{equation}
Now $\hat{\varphi}\in \mathcal{S}(\mathbf{R})$ and since $q_i \in \mathcal{E}'[0,1] \hookrightarrow \mathcal{S}'(\mathbf{R})$, we can apply the Fourier inversion theorem on $\mathcal{S}'(\mathbf{R})$ to find
\[  \lim_{\epsilon\rightarrow 0} \l Q_i^\epsilon, \varphi \r = \frac{1}{2\pi}\l \hat{q}_i, (\hat{\varphi})\check{\,}\,\r = \l q_i, \varphi \r. \]
Hence $\lim_{z\rightarrow \Gamma_i} Q = q_i$ in $\mathcal{S}'(\Gamma_i)$.
\end{proof}
\begin{remark}
By performing a similar calculation it is possible to determine the limit of $\partial Q/\partial\mathbf{n}$, where $\mathbf{n}$ is the normal to $\Gamma_i^\epsilon \simeq \Gamma_i$. In this case the corresponding terms for $\hat{q}_i$ in \eqref{weak} cancel and we find
\[  \lim_{\epsilon\rightarrow 0} \l (\partial Q/\partial \mathbf{n})_i^\epsilon, \varphi \r =  \l \partial_\mathbf{n} q_i, \varphi \r. \]
Thus, under the assumption that the global relation is satisfied, it is possible by using \eqref{rep}, to prove existence for the Dirichlet, the Neumann or the Robin boundary value problems.
\end{remark}

\begin{remark}
Central to this argument was the statement that our domain can be placed in the $z$-plane so that the vertex at $z=z_i$ lies at the origin. We certainly do not lose any generality by enforcing this situation. Indeed, suppose a priori that one is interested in the behaviour of the boundary values at a particular edge of the domain boundary. Then, \emph{after} this edge has been singled out, the domain can be embedded into $\mathbf{C}\simeq \mathbf{R}^2$ in any way we find convenient. Alternatively, one can check directly that the integral representation \eqref{rep} is invariant under rotations and translations in the $z$-plane.

This highlights an important advantage of the unified method introduced in \cite{fokas1997unified,fokas2000integrability}: in all boundary value problems there is an underlying gauge freedom, namely, the domain on which the boundary value problem is defined may be embedded into $\mathbf{R}^n$ a variety of ways. Clearly, the particular choice of gauge does not affect the analysis of the given boundary value problem, since the solution at a point $\mathbf{x}\in\Omega$ is determined by the boundary data prescribed on $\partial\Omega$, and the position of the point $\mathbf{x}$ \emph{relative to the boundary}.

In the classical approaches (Green's functions, boundary integral equations, etc.), the analysis is carried out in the \emph{physical} space. In this setting, gauge freedom gives no particular advantage (the action of a gauge transformation simply changes all variables appearing in the equations). However, in the unified approach the analysis is done entirely in the \emph{spectral} space. The spectral parameter is decoupled from the physical parameter, and now a choice of gauge \emph{does} affect the analysis of the spectral problem. In particular, as we have seen in this section, a choice of gauge allows one to single out a particular neighbourhood of the boundary $\partial\Omega$. Since the underlying problem is gauge invariant, these local arguments can be extended to all of $\partial\Omega$.

From the above discussion, it follows that the ``without loss of generality'' statements made in this section can be interpreted as an exploitation of gauge freedom: we have made a particular choice of gauge so that the analysis of the boundary values on part of our domain is simple, then we have extended these results to the entire domain using the underlying gauge freedom. 
\end{remark}

\section{Regularity of Boundary Data}
In this section we use the global relation \eqref{gr1} to derive regularity results for the unknown Neumann data. On the \emph{assumption} of a distributional solution to the global relation, we use elementary results from the theory of distributions to retrieve the classical regularity results for the Dirichlet-Neumman map for elliptic problems. The basic result is summarised below
\begin{theorem}\label{regthm}
Suppose the distributions $q_i,\partial_\mathbf{n}q_i \in \mathcal{E}'[0,1]$ ($i=1,\ldots,n$) are such that the global relation \eqref{gr1} is satisfied:
\begin{equation} \sum_{i=1}^n  \rho_i(\lambda) =0 \label{grdist} \end{equation}
where
\begin{align*}
\frac{\rho_i(\lambda)}{ |\Gamma_i|} &= \left\langle \partial_\mathbf{n} q_i, \psi_i^*\!\left(e^{-\ii\lambda z' + \ii \beta^2 \bar{z}'/\lambda}\right)\right\rangle \\ 
& \quad + \left(\lambda e^{\ii\alpha_i} + \frac{\beta^2}{\lambda e^{\ii\alpha_i} }\right) \left \langle  q_i, \psi_i^*\!\left(e^{-\ii\lambda z' + \ii \beta^2 \bar{z}'/\lambda}\right)\right\rangle.  \end{align*}
Suppose that the natural extension of $u_i$ belongs to $H^s(\mathbf{R})$ for some $1\leq i\leq n$ and $s\in \mathbf{R}$. Then $\partial_\mathbf{n} q_i \in H^{s-1}_{\mathrm{loc}}(0,1)$.
\end{theorem}
\begin{figure}\label{Regdiagram}
\begin{center}
\includegraphics[scale=1.2]{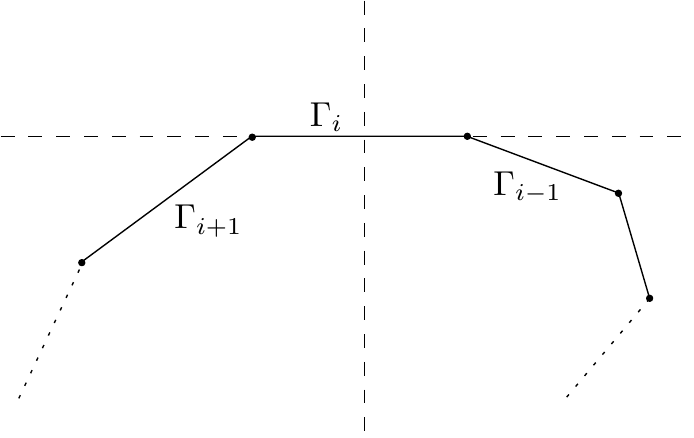}
\end{center}
\caption{Choice of gauge so that $\Gamma_i$ is aligned with the real $z$-axis.}\label{gaugechoice}
\end{figure}
The main idea in the proof is to use the gauge freedom to align the particular side we are interested in with the real $z$-axis, so that all other sides of the polygon are contained in the lower half plane, as in Figure \ref{gaugechoice}. The term in the global relation corresponding to the side aligned with the real $z$-axis, is closely related to the Fourier transform of the boundary values on this side. By proving that the other terms in the global relation decay in a suitable sense, the regularity of the unknown boundary values on the side aligned with the real $z$-axis can be deduced from the regularity of the known boundary data using standard properties of the Fourier transform and Sobolev spaces. The proof of the theorem requires two technical lemmas, which show that if we are interested in the regularity of the boundary values on the side $\Gamma_i$, the only important term in the global relation is $\rho_i$.
\begin{lemma}\label{distlem}
Let $\Omega$ be scaled and positioned so that it resides in the lower half $z$-plane and the side $\Gamma_i$ coincides with the interval $(0,1)$ on the real axis. Then
\begin{equation} \rho_i(\lambda) = (\partial_\mathbf{n} q_i)\hat{\,} (\lambda -\beta^2/\lambda) - \left( \lambda + \frac{\beta^2}{\lambda}\right) \hat{q}_i (\lambda-\beta^2/\lambda). \label{fouriernote}\end{equation}
Then the following estimate is valid for $\lambda >0$:
\begin{equation} \left| \rho_{i-1}(\lambda) + \rho_i(\lambda) + \rho_{i+1}(\lambda)\right| \leq C\left(\lambda+\frac{\beta^2}{\lambda}\right)^M e^{-\epsilon (\lambda + \beta^2/\lambda)} , \label{fourierestimate}\end{equation}
for some fixed $\epsilon>0$ and constants $C,M>0$.
\end{lemma}
\begin{proof}
The equality in \eqref{fouriernote} follows immediately from the fact that $\Gamma_i$ coincides with $(0,1)$ and $\alpha_i=\pi$. If $j\notin \{i,i\pm 1\}$, then the basic semi-norm estimate \eqref{seminorms} gives
\[ |\rho_j(\lambda)| \leq  C\left(\lambda+ \frac{\beta^2}{\lambda}\right)^N \max_{z\in\mathrm{cl}(\Gamma_j)}e^{(\lambda+\beta^2/\lambda)\Im z} \leq  C\left(\lambda+ \frac{\beta^2}{\lambda}\right)^N e^{-\epsilon(\lambda+\beta^2/\lambda)}\]
since $\Im z' \leq -\epsilon$ on $\Gamma_j$ for some $\epsilon>0$, using the fact that all the sides lie lower half plane. The estimate in \eqref{fourierestimate} follows from the global relation \eqref{grdist}
\[ \rho_{i-1}(\lambda) + \rho_i(\lambda) + \rho_{i+1}(\lambda) = -\sum_{j\neq i,i\pm 1} \rho_j(\lambda)  \]
and the previous estimate for the terms on the right hand side.
\end{proof}
\begin{lemma}
For $\varphi \in C^\infty_c(0,1)$ we define the operation $\star: \mathcal{S}(\mathbf{R})\times C^\infty(\mathbf{R}^+) \rightarrow C^\infty (\mathbf{R})$
\[ \hat{\varphi} \star \rho (\lambda) = \int_0^\infty \hat{\varphi}\left(\lambda  - \mu + \tfrac{\beta^2}{\mu}\right) \rho(\mu)\, \dd M(\mu), \]
where $\dd M(\mu) =(1+\beta^2/\mu^2)\dd \mu$. Then $\hat{\varphi} \star \rho_i \in \mathcal{S}(\mathbf{R})$.
\end{lemma}
\begin{proof}
First note that since $\varphi\in C^\infty_c(0,1)$, we have $\hat{\varphi}\in\mathcal{S}(\mathbf{R})$ and then it is straightforward to show that
\[ \mu\mapsto \varphi\left(\lambda - \mu + \tfrac{\beta^2}{\mu}\right) \in \mathcal{S}(\mathbf{R}^+). \]
This also means that this map is well defined, because
\[ \lim_{\mu\rightarrow 0}\left(1+\frac{\beta^2}{\mu^2} \right) \hat{\varphi}\left(\lambda - \mu + \tfrac{\beta^2}{\mu}\right) =0. \]
The estimate in \eqref{fourierestimate} and the closure of $\mathcal{S}(\mathbf{R})$ under convolution imply
\[ \hat{\varphi} \star \rho_i  = -\hat{\varphi}\star \rho_{i-1} - \hat{\varphi}\star \rho_{i+1} \,\,\, \textrm{mod $\mathcal{S}(\mathbf{R})$.} \]
We need to show that the terms on the right hand side of this expression are Schwartz functions. The analysis is similar for both terms, so we focus attention on $\hat{\varphi}\star\rho_{i+1}$. Our object of interest is given by
\begin{align} \hat{\varphi} \star \rho_{i+1} (\lambda) &= \int_0^\infty  \hat{\varphi}\left(\lambda  - \mu + \tfrac{\beta^2}{\mu}\right) \rho_{i+1}(\mu)\, \dd M(\mu) \nonumber \\
 &= \lim_{\epsilon\rightarrow 0} \int_0^\infty \hat{\varphi}\left(\lambda - \mu + \tfrac{\beta^2}{\mu}\right) \rho_{i+1}(\mu) e^{-\epsilon (\mu+\beta^2/\mu)}\, \dd M(\mu), \label{regularisation}
\end{align}
where we have inserted a regularisation for later convenience. This is justified since $\rho_{i+1}(\lambda)$ is certainly polynomially bounded (by the estimate in lemma \ref{distlem}) and $\hat{\varphi}\in \mathcal{S}(\mathbf{R})$, so the dominated convergence theorem applies. Now one can show that any distribution in $\mathcal{E}'[0,1]$ can be written as a finite sum of derivatives of bounded functions (see appendix). Employing this result we may assume that there exist bounded functions $\{g_j\}_{j=1}^N$ with supports contained in $[0,1]$ such that
\[ \left\langle \partial_\mathbf{n} q_{i+1}, \psi_{i+1}^*\left( e^{-\ii\lambda z + \ii \beta^2 \bar{z}/\lambda}\right)\right\rangle = \sum_{j=0}^N\int_0^1 g_j (\tau) \left(\frac{\dd}{\dd \tau}\right)^j \left( e^{-\ii\lambda \tau z_{i+2}  + \ii \beta^2\tau \bar{z}_{i+2}/\lambda}\right) \dd \tau, \]
and similarly for the term corresponding to $q_{i+1}$. Hence a generic term appearing in $\rho_{i+1}(\lambda)$ can be represented in the form
\[ \lambda^\beta \int_0^1 f(\tau) e^{-\ii\lambda \tau z_{i+2}  + \ii \beta^2\tau \bar{z}_{i+2}/\lambda}\, \dd \tau, \]
where $f$ is some bounded function and $\beta\in\mathbf{Z}$. Since $\rho_{i+1}(\lambda)$ is just a finite sum of such terms, it is sufficient for our analysis to use this expression instead of $\rho_{i+1}(\lambda)$. Applying Fubini's theorem, equation \eqref{regularisation} reads
\begin{align} & \hat{\varphi}\star\rho_{i+1} (\lambda)\nonumber \\  &\quad = \int_0^\infty \hat{\varphi}\left(\lambda  - \mu + \tfrac{\beta^2}{\mu}\right) \left[\mu^\beta \int_0^1 f(\tau) e^{-\ii\mu \tau z_{i+2}  + \ii \beta^2\tau \bar{z}_{i+2}/\mu}\, \dd \tau \right]\dd M(\mu) \nonumber \\
 &\quad = \lim_{\epsilon\rightarrow 0} \int_0^1  e^{-\ii\lambda y} \varphi (y) F_\epsilon(y)\, \dd y, \label{smooth}
\end{align}
where
\begin{align*} F_\epsilon(y) &= \int_0^1 f(\tau) \left[ \int_0^\infty \mu^\beta e^{-\mu (\ii(\tau z_{i+2} - y) + \epsilon)} e^{-\beta^2 (-\ii (\tau \bar{z}_{i+2} - y) + \epsilon)/\mu} \dd M(\mu) \right] \dd \tau. \\
 &\equiv \int_0^1 f(\tau) \left[ \int_0^\infty \mu^\beta e^{-\mu z_\epsilon(\tau,y)} e^{-\beta^2 \bar{z}_\epsilon (\tau,y)/\mu} \dd M(\mu) \right] \dd \tau,
\end{align*}
with $z_\epsilon(\tau,y)$ defined accordingly. We now proceed to show that the function $F(y) = \lim_{\epsilon\rightarrow 0} F_\epsilon(y)$ is smooth on the interval $(0,1)$. Making a change of variables in the $\mu$-integral
\[ \mu' = \mu z_\epsilon(\tau,y), \]
and noting that $\Re z_\epsilon(\tau,y) \geq \epsilon$, we can use elementary techniques from complex analysis to show that the corresponding contour of integration in the $\mu'$-plane can be rotated back to the real axis. This gives
\[ F_\epsilon(y) = \int_0^1 \frac{f(\tau)}{z_\epsilon(\tau,y)^{\beta+1}} \left[ \int_0^\infty (\mu')^\beta e^{-\mu'} e^{-\beta^2|z_\epsilon(\tau,y)|^2/\mu'} \dd M(\mu'/z_\epsilon) \right] \dd x. \]
For any $y\in (0,1)$ there exists a $\delta_y>0$ such that $|z_0(\tau,y)|\geq \delta_y$ for $\tau \in [0,1]$ so we can take $\epsilon\rightarrow 0$ for $y$ in this open interval, although the convergence is not uniform in $y$. Also, for any $y\in (0,1)$ the $\mu'$ integral decays exponentially so we can differentiate under the integral sign as many times as needed. It follows that $F \in C^\infty (0,1)$ and hence $\varphi F \in C_c^\infty(0,1)$. We deduce from \eqref{smooth} that $\hat{\varphi}\star \rho_{i+1} \in \mathcal{S}(\mathbf{R})$, since it is the Fourier transform of a smooth function with compact support. The argument showing $\hat{\varphi}\star \rho_{i-1}\in \mathcal{S}(\mathbf{R})$ follows in a similar fashion.
\end{proof}
\begin{proof}[Proof of Theorem \ref{regthm}]
Using the result of the previous lemma we have
\[ \hat{\varphi}\star \rho_i \in \mathcal{S}(\mathbf{R}) \qquad \forall \varphi \in C^\infty_c(0,1). \]
We introduce the new variable $k=\mu-\beta^2/\mu$ for the integral defining $\hat{\varphi}\star \rho_i$. It is clear that $k(\mu):\mathbf{R}^+\rightarrow\mathbf{R}$ is a diffeomorphism and $\dd k = \dd M(\mu)$. Using this observation in \eqref{fouriernote} we find
\[ \rho_i(\mu(k)) = (\partial_\mathbf{n} q_i)\hat{\,}(k) - \sqrt{k^2+4\beta^2}\hat{q}_i (k). \]
Making the substitution in the integral for $\hat{\varphi}\star \rho_i$ we find
\[ \int_{-\infty}^\infty \hat{\varphi}(\lambda-k) (\partial_\mathbf{n} q_i)\hat{\,}(k)\, \dd k - \int_{-\infty}^\infty \hat{\varphi}(\lambda-k)\sqrt{k^2+4\beta^2} \hat{q}_i(k)\, \dd k=0 \,\,\, \mathrm{mod}\, \mathcal{S}(\mathbf{R}). \]
The first term on the left hand side it just $\hat{\varphi}*(\partial_\mathbf{n} q_i)\hat{\,}$, so
\begin{align} (\varphi \partial_\mathbf{n} q_i)\hat{\,}(\lambda) &= \frac{1}{2\pi} (\hat{\varphi} * (\partial_\mathbf{n} q_i)\hat{\,})(\lambda) \nonumber \\
&= \frac{\ii\beta^2}{ \pi} \int_{-\infty}^\infty \hat{\varphi}(\lambda -k)\sqrt{k^2+4\beta^2} \hat{q}_i(k)\, \dd k\,\,\, \mathrm{mod}\,\mathcal{S}(\mathbf{R}). \label{conv} 
\end{align}
Using the standard notation $\l \lambda \r = (1+\lambda^2)^{1/2}$, it follows that
\[ \l k \r \simeq (k^2 + 4\beta^2)^{1/2}. \]
We want to estimate the $L^2$ norm of the right hand side of \eqref{conv} after multiplication by $\l \lambda \r^{s-1}$, i.e. we want to estimate the $L^2$ norm of
\[  \int_{-\infty}^\infty \frac{ \l \lambda \r^{s-1} }{\l \lambda -k \r^{s-1}}|\hat{\varphi}(k)| \l \lambda - k \r^s |\hat{q}_i(\lambda-k)|\, \dd k. \]
Using Peetre's inequality
\[ \frac{\l \lambda \r^{s-1}}{\l \lambda - k\r^{s-1}} \leq 2^{|s-1|} \l k \r^{s-1}, \]
and applying Young's $L^2$ inequality we find
\[ \left\| \int_{-\infty}^\infty \frac{ \l \lambda \r^{s-1} }{\l \lambda -k \r^{s-1}}|\hat{\varphi}(k)| \l \lambda - k \r^s |\hat{q}_i(\lambda-k)|\, \dd k \right\|_{L^2} \leq \| \varphi \|_{H^{s-1}} \| q_i \|_{H^s}. \]
Using this estimate in \eqref{conv} we find that $\varphi \partial_\mathbf{n} q_i \in H^{s-1}(\mathbf{R})$. Since $\varphi\in C^\infty_c(0,1)$ is arbitrary, we conclude that $\partial_\mathbf{n} q_i \in H^{s-1}_\mathrm{loc} (0,1)$.
\end{proof}
\begin{remark}
The proof is almost the same if one assumes that $\varphi \in \mathcal{S}(0,1)$. The only modification needed is to take into consideration that the function $F(y)=\lim_{\epsilon\rightarrow 0} F_\epsilon(y)$ has a singularities of finite order at $y=0$. It is then easy to show that $\varphi F \in \mathcal{S}(0,1)$, so its Fourier transform belongs to $\mathcal{S}(\mathbf{R})$ and the proof follows as in Theorem \ref{regthm}.
\end{remark}

We note that this proof can easily be modified to deal with more general Sobolev spaces $W^{s,p}(\Gamma_i)$ where $p\in(1,\infty)$. Here we only state the relevant definitions and results, and refer the reader to \cite{hormander1985analysis3} for a comprehensive account.

The class of Fourier multipliers $S^n(\mathbf{R})$ are defined by the smooth functions $p(x,\lambda)$ such that
\[ | \partial^\beta_x \partial^\alpha_\lambda p(x,\lambda)| \leq C_{\alpha\beta} \langle \lambda \rangle^{n- |\alpha| } .\]
For $p\in S^n(\mathbf{R})$ we define the pseudo-differential operator $p(x,D)\in \Psi^n(\mathbf{R})$ via
\[ p(x,D) f(x) = \int_{-\infty}^\infty p(x,\lambda) \hat{f}(\lambda) e^{\ii \lambda x}\, \dd \lambda, \]
for Schwartz functions $f$. The action of $p(x,D)$ can then be defined on $\mathcal{S}'(\mathbf{R})$ by duality. A fundamental property of the pseudo-differential operators $\Psi^{n}(\mathbf{R})$ is that for $q\in (1,\infty)$ the following map is continuous,
\[ (p,f) \in S^{n}(\mathbf{R}) \times W^{s,q}(\mathbf{R}) \mapsto p(x,D)f \in W^{s-n, q}_{\mathrm{loc}}(\mathbf{R}), \]
meaning that for $p\in S^n(\mathbf{R})$ and $f\in W^{s,q}(\mathbf{R})$ with $q\in (1,\infty)$, we have
\[ \| p(x,D) f\|_{W^{s-n,q}_{\mathrm{loc}}} \leq C \| f\|_{W^{s,q}}, \]
for some constant $C$. We say that $p(x,D) \in \Psi^n(\mathbf{R})$ is elliptic of order $n$ if
\[ |p(x,\lambda)| \geq C \langle \lambda \rangle ^{-n}, \]
for $|\lambda|$ sufficiently large. With these definitions we have the following result.
\begin{theorem}
Let $\{\partial_\mathbf{n} q_i\}_{i=1}^n$ and $\{q_i\}_{i=1}^n$ be distributional solutions of \eqref{grdist}. Then, there exists an elliptic pseudo-differential operator $p(D)\in \Psi^{1}(\mathbf{R})$, such that for each $\varphi \in C^\infty_c(0,1)$
\[ \varphi \left( \partial_\mathbf{n} q_i - p(D) q_i\right) =0\,\,\, \mathrm{mod}\,\mathcal{S}(\mathbf{R}). \]
\end{theorem}
\begin{proof}
Recall the equality \eqref{conv}
\begin{equation} (\varphi \partial_\mathbf{n} q_i)\hat{\,}(\lambda)=  \frac{\ii\beta^2}{ \pi} \int_{-\infty}^\infty \hat{\varphi}(\lambda -k)\sqrt{k^2+4\beta^2} \hat{q}_i(k)\, \dd k\,\,\, \mathrm{mod}\,\mathcal{S}(\mathbf{R}), \label{conv'}\end{equation}
is valid for all $\varphi\in C^\infty_c(0,1)$. We introduce the function
\[ p (k) = 2\ii \beta^2\sqrt{k^2+4\beta^2} . \]
It is clear that $p(D) \in \Psi^1(\mathbf{R})$ is elliptic. Working modulo the space of Schwartz functions, the equivalence in \eqref{conv'} reads
\begin{align*} (\varphi \partial q_i)\hat{\,}(\lambda) &= \frac{1}{2\pi}\int_{-\infty}^\infty \varphi(\lambda-k) p(k) \hat{q}_i(k)\, \dd k\\
 &= \frac{1}{2\pi}\hat{\varphi} * (p(D)q)\hat{\,}(\lambda) \\
 &= (\varphi p(D) q)\hat{\,}(\lambda).
\end{align*}
So we certainly have
\[  (\varphi \partial_\mathbf{n} q_i - \varphi p(D) q_i)\hat{\,}(\lambda) = 0 \,\,\, \mathrm{mod}\,\mathcal{S}(\mathbf{R}). \]
The result follows from the fact the Fourier transform is an isomorphism on $\mathcal{S}(\mathbf{R})$.
\end{proof}

\section{Corner Singularities}\label{cornersingfokas}
We consider the existence of corner singularities for solutions of the modified Helmholtz equation in a polygonal domain. The basic differential form for the Modified Helmholtz equation is given by
\begin{equation} W(z,\bar{z};\lambda) = e^{-\ii\lambda z+\ii\beta^2 \bar{z}/\lambda}\left[ \left( \frac{\partial q}{\partial z}+\ii \lambda q\right) \dd z + \left(\frac{\ii\beta^2}{\lambda} q-\frac{\partial q}{\partial \bar{z}} \right) \dd\bar{z}\right]. \label{fok1} \end{equation}
The associated global relation is given by $\oint_{\partial\Omega} W=0$. We concentrate on the particular corner $z_i$. Introducing the local coordinates $(\rho,\theta)$ defined by $z-z_i = \rho e^{\ii \theta}$, and using
\begin{equation} \frac{\partial}{\partial z} = \frac{e^{-\ii \theta}}{2} \left( \frac{\partial}{\partial \rho} - \frac{\ii}{\rho} \frac{\partial}{\partial\theta} \right), \quad \quad \frac{\partial}{\partial \bar z} = \frac{e^{\ii \theta}}{2} \left( \frac{\partial}{\partial \rho} + \frac{\ii}{\rho} \frac{\partial}{\partial\theta} \right), \label{fok2}\end{equation}
we find that
\begin{equation} W = \ii e^{-\ii\lambda z(\rho,\theta)+\ii\beta^2 \bar{z}(\rho,\theta)/\lambda} \left[ -\frac{1}{\rho} \frac{\partial q}{\partial \theta} + \left( \lambda e^{\ii \theta} + \frac{\beta^2}{\lambda e^{\ii\theta}} \right) q \right] \dd \rho. \label{fok3} \end{equation}
Without loss of generality we position the polygonal domain in the complex $z$-plane so that $z_i=\ii$, and $\Im z_j <0$ for $j\neq i$ as in Figure \ref{cornerdomain}.
\begin{figure}\label{cornerdomain}
\begin{center}
\includegraphics[scale=1.0]{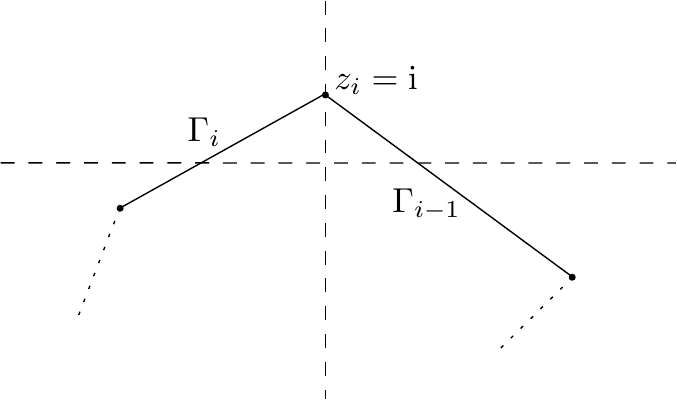}
\end{center}
\caption{\small Gauge choice for corner singularity analysis.}
\end{figure}
Then, for $z$ on any side other than on the sides $\Gamma_i$, $\Gamma_{i-1}$, and for $\lambda$ real and positive, we have
\begin{equation} \left| e^{-\ii \lambda z + \ii \beta^2 \bar{z}/\lambda} \right| \leq e^{-\epsilon \lambda}, \quad z \in \partial \Omega \setminus ( \Gamma_i \cup \Gamma_{i-1}), \label{fok4} \end{equation}
for some $\epsilon>0$. Letting $z_i=\ii$, we find
\begin{equation} e^{-\ii \lambda z + \ii \beta^2 \bar{z}/\lambda} = e^{\lambda + \beta^2/\lambda} e^{-\rho k(\lambda,\theta)}, \qquad k(\lambda,\theta) \defn \lambda e^{\ii (\theta+\pi/2)} + \frac{\beta^2}{\lambda e^{\ii(\theta+\pi/2)}}. \label{fok5} \end{equation}
For both $\Gamma_i$ and $\Gamma_{i-1}$, $\Re k(\lambda,\theta)>0$, since
\[ -\pi < \theta < 0, \quad \textrm{i.e.} \quad |\theta+\pi/2| < \pi/2. \]
The sides $\Gamma_{i-1}$ and $\Gamma_i$ yield the following contribution to the global relation:
\begin{multline}
\ii e^{\lambda + \beta^2/\lambda} \left[ \int_0^{|\Gamma_i|} e^{-\rho k(\lambda,-\theta_i)} \left[ \frac{1}{\rho} \frac{\partial q}{\partial \theta} (\rho,-\theta_i) + \right.\right. \\
 \left.\left. \left( \lambda e^{-\ii \theta_i} + \frac{\beta^2}{\lambda e^{-\ii\theta_i}} \right) q(\rho,-\theta_i) \right] \dd \rho \right]  \\
-\ii e^{\lambda + \beta^2/\lambda}\left[  \int_0^{|\Gamma_{i-1}|}e^{-\rho k(\lambda,-\theta_{i-1})} \left[ \frac{1}{\rho} \frac{\partial q}{\partial \theta} (\rho,-\theta_{i-1}) + \right.\right. \\
 \left.\left.  \left( \lambda e^{-\ii \theta_{i-1}} + \frac{\beta^2}{\lambda e^{-\ii\theta_{i-1}}} \right) q(\rho,-\theta_{i-1}) \right] \dd \rho  \right]. \label{fok6}
\end{multline}
We can examine the possibility of corner singularities by considering the limit of equation \eqref{fok6} as $\lambda\rightarrow \infty$. I what follows we analyse several types of boundary conditions.

\subsection{Continuous Neumann-Neumann}
Suppose that Neumann boundary conditions $\rho^{-1}\partial q/\partial \theta(\rho,-\theta_\ell)$ for $\ell=i-1,i$ are prescribed at the sides $\Gamma_{i-1}$ and $\Gamma_i$. Furthermore, assume that the given data are continuous at the corner. Let
\[ q(\rho,-\theta_\ell) \sim D_\ell \rho^{d_\ell} \quad \textrm{as $\rho\rightarrow 0$,} \quad \ell = i,i-1, \]
where $d_\ell>-1$ for $\ell=i,i-1$, so that the integrals in \eqref{fok6} are well-defined. Employing Watson's lemma\footnote{Recall Watson's lemma $\int_0^1 e^{-\nu \rho} \rho^\ell \, \dd \rho \sim \Gamma(\ell+1) \nu^{-1-\ell}$ as $\nu \rightarrow \infty$.} in \eqref{fok6} we find
\begin{equation} \frac{\lambda e^{-\ii \theta_i} D_i \Gamma(d_i +1)}{\left( \lambda e^{\ii(-\theta_i + \pi/2)}\right)^{d_i+1}} = \frac{ \lambda e^{-\ii \theta_{i-1}} D_{i-1} \Gamma(d_{i-1}+1)}{\left( \lambda e^{\ii (-\theta_{i-1}+\pi/2)}\right)^{d_{i-1}+1}}. \label{fok9} \end{equation}
Hence,
\[ d_{i-1}=d_i, \quad \frac{D_i}{D_{i-1}} = e^{-\ii \Delta_{i,i-1} d_i}. \]
Note that convexity implies that $\Delta_{i,i-1} = \alpha_i - \alpha_{i-1} \in (0,\pi)$. We deduce that there exists an integer $M\in \{0,1,2,\ldots\}$, such that
\begin{equation} d_{i-1}=d_i = \frac{2\pi M}{\Delta_{i,i-1}},  \quad D_{i-1} = D_i. \label{fok10} \end{equation}
The inequalities $d_\ell >-1$ for $\ell=i,i-1$ imply that $M$ cannot be negative. We deduce that either both $q(\rho,-\theta_\ell)$ for $\ell=i,i-1$ are $\mathcal{O}(1)$ as $\rho\downarrow 0$ or
\[ d_{i-1} = d_i > 2. \]
In either case, corner singularities do not occur, independent of the internal angle $\Delta_{i,i-1}$.

\subsection{Continuous Dirichlet-Dirichlet}
Suppose that Dirichlet boundary conditions $q(\rho,-\theta_{\ell})$, $\ell=i,i-1$ are prescribed at the sides $\Gamma_i$ and $\Gamma_{i-1}$. Furthermore, assume that the given data are continuous at the corner. Let
\[ \frac{\partial q}{\partial \theta} (\rho,-\theta_\ell) \sim N_\ell \rho^{n_\ell} \quad \textrm{as $\rho\rightarrow 0$,} \quad \ell=i,i-1, \]
where $n_\ell>0$ for $\ell=i,i-1$, so the integrals in \eqref{fok6} are well-defined. The leading order term in $\lambda$ as $\lambda \rightarrow \infty$ originating from the Dirichlet data vanishes due to continuity. Indeed, this term gives rise to the contribution, as $\lambda\rightarrow \infty$,
\[ \ii \Big( q(0,-\theta_i) - q(0,-\theta_{i-1})\Big), \]
which vanishes. Employing Watson's lemma in \eqref{fok6} we find
\[ \frac{ N_i \Gamma(n_i) }{\left( \lambda e^{\ii(-\theta_i + \pi/2)} \right)^{n_i}} = \frac{N_{i-1} \Gamma(n_{i-1})}{\left( \lambda e^{\ii (-\theta_{i-1} + \pi/2)}\right)^{n_{i-1}}}. \]
Thus, there exists an integer $M\in \{1,2,\ldots\}$, such that
\[ n_{i-1} = n_i = \frac{2\pi M}{\Delta_{i,i-1}}, \quad N_{i-1} = N_i. \]
Employing similar reasoning as in the previous section we deduce that corner singularities do not arise, independent of the internal angle $\Delta_{i,i-1}$.

\subsection{Vanishing at the corner: Dirichlet-Neumann}
Suppose that Neumann and Dirichlet boundary conditions are prescribed at the side $\Gamma_i$ and $\Gamma_{i-1}$ respectively. Furthermore, assume that the given data vanish at the corner. Let
\[ q(\rho,-\theta_i) \sim D_i \rho^{d_i}, \quad \frac{\partial q}{\partial \theta} (\rho,-\theta_{i-1}) \sim N_{i-1} \rho^{n_{i-1}} \quad \textrm{as $\rho\rightarrow\infty$,} \]
where $d_i>-1$ and $n_{i-1}>0$. Employing Watson's lemma in \eqref{fok6} we find
\[ \frac{ \lambda e^{-\ii\theta_i} D_i \Gamma( d_i+1)}{\left( \lambda e^{\ii(-\theta_i + \pi/2)}\right)^{d_i+1}} = \frac{ N_{i-1} \Gamma(n_{i-1})}{ \left( \lambda e^{\ii(-\theta_{i-1} + \pi/2)}\right)^{n_{i-1}}}. \]
Hence
\[ d_i = n_{i-1}, \quad \frac{D_i d_i}{N_{i-1}} = e^{\ii (\pi/2 - \Delta_{i,i-1} d_i)}. \]
We deduce that there exists an $M\in \{0,1,2,\ldots \}$ such that
\[ N_{i-1} =D_i d_i , \quad d_i=n_{i-1} = \frac{\pi}{\Delta_{i,i-1}} \left( 2M + \frac{1}{2}\right). \]
Singularities can occur only if $n_{i-1} < 1$, i.e. if $\Delta_{i,i-1} > \pi/2$. We conclude that corner singularities can occur only if there is an internal angle smaller larger than $\pi/2$, in agreement with the classical results for elliptic boundary value problems in corner domains \cite{dauge1988elliptic}.

\subsection{Discontinuous Dirichlet-Dirichlet}
In this case, as $\lambda\rightarrow \infty$, there exists a non-vanishing term proportional to the discontinuity. This implies the existence of non-integrable singularities.

\begin{remark}
The limit as $\lambda \rightarrow 0$ yields precisely the same information as the limit $\lambda \rightarrow \infty$. This is not surprising, since \eqref{fok6} is invariant under the transformation $\lambda \mapsto \beta^2/\lambda$ followed by complex conjugation.
\end{remark}

\section{Non-Integrable Singularities}
It was shown in \S\ref{cornersingfokas} that discontinuous Dirichlet boundary conditions can give rise to non-integrable singularities. In order to eliminate such a singularity at the corner $z_i$, we introduce the differential form $V$ defined by
\[ V(\rho,\theta,\lambda) = W(\rho,\theta,\lambda) + \dd \left( \ii e^{\lambda + \beta^2/\lambda} e^{-\rho k(\lambda,\theta)} \int_{\rho^*}^\rho \frac{1}{\rho'} \frac{\partial q}{\partial \theta} (\rho',\theta)\, \dd \rho' \right), \]
where $\rho^* \neq 0$. Along a ray of constant $\theta$, $V$ is given by
\begin{multline} V(\rho,\theta,\lambda) = \ii e^{\lambda+\beta^2/\lambda} e^{-\rho k(\lambda,\theta)} \left[ -\ii \left( \lambda e^{\ii \theta} - \frac{\beta^2}{\lambda e^{\ii\theta}}\right) \int_{\rho^*}^\rho \frac{1}{\rho'} \frac{\partial q}{\partial \theta} (\rho',\theta)\, \dd \rho'\right. \\
 \left. + \left( \lambda e^{\ii\theta} + \frac{\beta^2}{\lambda e^{\ii\theta}}\right) q(\rho,\theta) \right]\dd \rho \label{fokas5.1} \end{multline}
and hence the non-integrable singularity is eliminated. A boundary value problem involving several non-integrable singularities, can be decomposed into a series of problems each involving a single corner. In what follows we analyse in detail a particular boundary value problem which has two non-integrable singularities. In this case it is possible to eliminate both of these singularities using a single differential form.

\subsection*{Example} Let $q(x,y)$ satisfy the modified Helmholtz equation in the interior of the half-strip formed by the corners
\[ z_1 = \infty + \ii \ell, \quad z_2 = \ii \ell, \quad z_3 =0, \quad z_4 = \infty, \]
where $\ell$ is real and positive. Let
\begin{equation} q(x,0) = q(x,\ell) = 0, \quad 0<x<\infty; \quad q(0,y) = 1, \quad 0<y<\ell. \end{equation}

The symmetry relation $q(x,y)=q(x,\ell-y)$ implies
\[ \frac{\partial q}{\partial y} (x,0) = - \frac{\partial q}{\partial y} (x,\ell), \quad \frac{\partial q}{\partial y} (x, \tfrac{1}{2} \ell) =0, \quad 0<x<\infty. \]
It is possible to eliminate the non-integrable singularities at the corners $z_2$ and $z_3$ using the following differential form:
\begin{equation} V = W + \dd \left[ e^{\Omega x + \omega y} c_1 \int_\infty^x \frac{\partial q}{\partial y}(x',y)\, \dd x' +  e^{\Omega x + \omega y} c_2 \int_{\ell/2}^y \frac{\partial q}{\partial x} (x,y')\, \dd y' \right], \label{fokas5.6}\end{equation}
where
\[ \Omega(\lambda) = -\ii \left( \lambda - \frac{\beta^2}{\lambda} \right), \quad \omega(\lambda) = \lambda+\frac{\beta^2}{\lambda} \]
and $c_1,c_2$ are constants that satisfy $c_1 - c_2 =1$. Indeed, the definition of $W$ in \eqref{fok1} implies
\[ W= e^{\Omega x + \omega y} \left[ \left( - \frac{\partial q}{\partial y} + \omega q	\right) \dd x + \left( \frac{\partial q}{\partial x} - \Omega q\right) \dd y \right]. \]
Using this equation and employing the identities
\begin{align*}
\int_{\ell/2}^y \frac{\partial^2 q}{\partial x^2} (x,y')\, \dd y' &= 4\beta^2 \int_{\ell/2}^y q(x,y')\, \dd y' - \frac{\partial q}{\partial y}(x,y), \\
\int_\infty^x \frac{\partial^2 q}{\partial y^2} (x',y)\, \dd x' &= 4\beta^2 \int_\infty^x q(x,y')\, \dd x' - \frac{\partial q}{\partial x}(x,y),
\end{align*}
equation \eqref{fokas5.6} becomes
\begin{align}
W = e^{\Omega x + \omega y} & \left\{   \left[ \Omega \left( c_1 \int_\infty^x \frac{\partial q}{\partial y}(x',y)\, \dd x' + c_2 \int_{\ell/2}^y \frac{\partial q}{\partial x}(x,y')\, \dd y' \right) \right. \right. \label{fokas5.9} \\
&\qquad\qquad\qquad\left. +\, 4\beta^2 c_2 \int_{\ell/2}^y q(x,y')\, \dd y' + \omega q(x,y) \right] \dd x \nonumber  \\
& +\left[ \omega \left( c_1 \int_\infty^x \frac{\partial q}{\partial y}(x',y)\, \dd x' + c_2 \int_{\ell/2}^y \frac{\partial q}{\partial x}(x,y')\, \dd y' \right) \right. \nonumber \\
&\qquad\qquad\qquad \left.\left. +\, 4\beta^2 c_1 \int_\infty^x q(x',y)\, \dd x' -\Omega q(x,y) \right] \dd y \right\}. \nonumber
\end{align}
Thus, the Neumann boundary values, which involve the non-integrable singularities, have been eliminated.

For simplicity we set $c_1=1$ and $c_2 =0$, thus, \eqref{fokas5.9} becomes
\begin{align}
W = e^{\Omega x + \omega y} & \left\{   \left[ \Omega \int_\infty^x \frac{\partial q}{\partial y}(x',y)\, \dd x'   + \omega q(x,y) \right]\right. \dd x \label{fokas5.10}  \\
& \left. +\left[ \omega  \int_\infty^x \frac{\partial q}{\partial y}(x',y)\, \dd x'  + 4\beta^2  \int_\infty^x q(x',y)\, \dd x' -\Omega q(x,y) \right] \dd y \right\}. \nonumber
\end{align}

Using this differential form and employing the method of \cite{fokas1997unified,fokas2000integrability}, it can be shown that $q$ is given by
\begin{multline} q(x,y) = -\frac{1}{2\pi} \left\{ \int_0^\infty e^{-\Omega x - \omega y} \frac{ G(\lambda)}{1+ e^{\omega \ell}} \frac{ \dd \lambda}{\lambda}  + \int_{\ii\infty}^0 e^{-\Omega x-\omega y} G(\lambda) \, \frac{\dd \lambda}{\lambda} \right. \\ \left. + \int_0^{-\infty} e^{-\Omega x + \omega (\ell-y)} \frac{G(\lambda)}{1+ e^{\omega \ell}} \frac{\dd\lambda}{\lambda} \right\}.  \label{fokas5.11} \end{multline}
where
\[ G(\lambda) = \Omega(\lambda) \frac{ e^{\omega(\lambda) \ell} -1}{\omega(\lambda)}. \]
In what follows we verify that the expression, as defined by \eqref{fokas5.11}, satisfies the given boundary conditions.

\subsubsection*{Boundary condition $q(x,0)=0$:} Evaluating equation \eqref{fokas5.11} at $y=0$ we find
\begin{multline} q(x,0) = -\frac{1}{2\pi} \left\{ \int_0^\infty e^{-\Omega x } \frac{ G(\lambda)}{1+ e^{\omega \ell}} \frac{ \dd \lambda}{\lambda}  + \int_{\ii\infty}^0 e^{-\Omega x} G(\lambda) \, \frac{\dd \lambda}{\lambda} \right. \\ \left. + \int_0^{-\infty} e^{-\Omega x + \omega \ell} \frac{G(\lambda)}{1+ e^{\omega \ell}} \frac{\dd\lambda}{\lambda} \right\}.  \nonumber \end{multline}
Noting that
\[ \int_0^{-\infty} e^{-\Omega x + \omega \ell} \frac{G(\lambda)}{1+ e^{\omega \ell}} \frac{\dd\lambda}{\lambda} \equiv \int_0^{-\infty} e^{-\Omega x } G(\lambda)\, \frac{\dd\lambda}{\lambda} - \int_0^{-\infty} e^{-\Omega x } \frac{G(\lambda)}{1+ e^{\omega \ell}} \frac{\dd\lambda}{\lambda} \]
we find
\[ q(x,0) = - \frac{1}{2\pi} \left( \int_{\ii\infty}^0 + \int_0^{-\infty} \right) e^{-\Omega x} G(\lambda) \, \frac{\dd\lambda}{\lambda}=0. \]

\subsubsection*{Boundary condition $q(x,\ell)=0$:} This follows in a similar fashion.

\subsubsection*{Boundary condition $q(0,y)=1$:} Evaluating \eqref{fokas5.11} at $x=0$ we find
\begin{multline} q(0,y) = -\frac{1}{2\pi} \left\{ \int_0^\infty e^{ - \omega y} \frac{ G(\lambda)}{1+ e^{\omega \ell}} \frac{ \dd \lambda}{\lambda}  + \int_{\ii\infty}^0 e^{-\omega y} G(\lambda) \, \frac{\dd \lambda}{\lambda} \right. \\ \left. + \int_0^{-\infty} e^{\omega (\ell-y)} \frac{G(\lambda)}{1+ e^{\omega \ell}} \frac{\dd\lambda}{\lambda} \right\}.  \label{fokas5.13} \end{multline}
In order to take care of the singularity at $\lambda = \ii \beta$, we deform the contour of the second integral on the right hand side of \eqref{fokas5.13} to the contour $L$ depicted in Figure 4.
\begin{figure}\label{contourL}
\begin{center}
\includegraphics[scale=0.8]{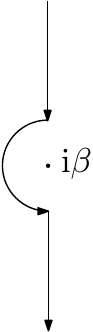}
\end{center}
\caption{Contour $L$ in the complex $\lambda$-plane.}
\end{figure}
Then, $q(0,y)$ can written in the form
\begin{multline} q(0,y) = -\frac{1}{2\pi} \left\{  \left( \int_L + \int_0^{-\infty}\right) e^{\omega (\ell -y)} \frac{\Omega}{\omega} \, \frac{\dd \lambda}{\lambda} - \left( \int_L + \int_0^{\infty} \right) e^{-\omega y} \frac{\Omega}{\omega}\, \frac{\dd \lambda}{\lambda} \right. \\
\left. - 2 \int_{-\infty}^\infty e^{-\omega y} \frac{e^{\omega \ell}}{1+e^{\omega \ell}} \frac{\Omega}{\omega}\, \frac{\dd \lambda}{\lambda} \right\}. \label{fokas5.14}  \end{multline}
Jordan's lemma implies that the first integral on the right hand side of \eqref{fokas5.14} vanishes. The integrand of the third integral remains invariant under the transformation $\lambda \mapsto 1/\lambda$, thus this integral also vanishes. The second integral on the right hand side of \eqref{fokas5.14} has a pole at $\lambda = \ii \beta$ with residue $-\ii$. Hence $q(0,y)=1$.

\subsubsection*{The non-integrable corner singularity}
Equation \eqref{fokas5.11} implies
\begin{multline*} \int_x^\infty \frac{\partial q}{\partial y}(x',0)\, \dd x' = \frac{1}{2\pi} \left\{ \int_0^\infty e^{-\Omega x} \left( \frac{ e^{\omega \ell} -1}{e^{\omega \ell} +1}\right) \frac{\dd \lambda}{\lambda} + \int_{\ii\infty}^0 e^{-\Omega x} (e^{\omega \ell} -1) \frac{\dd\lambda}{\lambda} \right. \\
\left. \quad + \int_0^{-\infty} e^{-\Omega x+\omega \ell} \left( \frac{ e^{\omega \ell} -1}{e^{\omega \ell} +1}\right) \frac{\dd \lambda}{\lambda} \right\}.
\end{multline*}
Using a similar argument to the one used in calculating $q(x,0)$ we find that the expression on the right hand side simplifies to
\begin{equation*} \int_x^\infty \frac{\partial q}{\partial y}(x',0)\, \dd x' =  \frac{1}{\pi} \int_0^\infty e^{-\ii (\lambda - \beta^2/\lambda)} \left( \frac{e^{(\lambda + \beta^2/\lambda)\ell} -1 }{ e^{(\lambda+\beta^2/\lambda)\ell} +1} \right) \frac{\dd \lambda}{\lambda}. \end{equation*}
We will show that
\begin{equation} \int_x^\infty \frac{\partial q}{\partial y}(x',0)\, \dd x' = - \frac{4}{\pi} \log x + \mathcal{O}(1) \quad \textrm{as $x\downarrow 0$.} \label{fokas5.17} \end{equation}
Indeed, letting $k = \lambda - \beta^2/\lambda$, i.e.
\[ \frac{\dd \lambda}{\lambda} = \frac{1}{\sqrt{ k^2+4\beta^2}}, \quad \lambda+\frac{\beta^2}{\lambda} = \sqrt{ k^2+4\beta^2}, \]
we find
\begin{multline} \int_x^\infty \frac{\partial q}{\partial y}(x',0)\, \dd x' = \frac{1}{\pi} \int_{-\infty}^\infty e^{-\ii k x} \left[ \frac{ e^{\ell\sqrt{ k^2+4\beta^2}}-1 }{e^{\sqrt{ k^2+4\beta^2}} +1} -1 \right] \frac{ \dd k}{\sqrt{ k^2+4\beta^2}} \\
+ \frac{1}{\pi} \int_{-\infty}^\infty \frac{e^{-\ii k x}}{\sqrt{ k^2+4\beta^2}}\, \dd k. \end{multline}
The integrand in the first term is $\mathcal{O}(1/k^2)$ as $k\rightarrow \infty$, so is absolutely integrable. Hence, the first term, being the Fourier transform of an element of $L^1(\mathbf{R})$, gives rise to a bounded, continuous function of $x$. The second integral is
\[ \frac{2}{\pi} \int_0^\infty \frac{\cos kx}{\sqrt{ k^2+4\beta^2} }\, \dd k \equiv \frac{2}{\pi} \int_0^1 \frac{ \cos k}{\sqrt{ k^2+4\beta^2x^2 }}\, \dd k + \frac{2}{\pi} \int_1^\infty \frac{ \cos k}{\sqrt{ k^2+4\beta^2x^2 }}\, \dd k. \]
It is straightforward to show that the second of these integrals is a bounded, continuous function of $x$. Writing $\cos k = (\cos k-1)+1$ we deduce
\[ \int_x^\infty \frac{\partial q}{\partial y}(x',0)\, \dd x' = \frac{2}{\pi} \int_0^1 \frac{\dd k}{\sqrt{k^2 + 4\beta^2 x^2}} + \mathcal{O}(1). \]
The claim in \eqref{fokas5.17} now follows by computing the elementary integral on the right hand side of the above equation.

\section{Appendix}
It was claimed in \S 4 that a distribution in $\mathcal{E}'[0,1]$ can be written as a finite sum of (weak)-derivatives of bounded functions supported on $[0,1]$. We have been unable to find the proof in the literature, so here we present a proof.
\begin{theorem*}
For each $u \in \mathcal{E}'[0,1]$ there exist a finite collection of complex Borel measures $\{ \mu_0, \ldots, \mu_n\} $ such that
\[ \l u, \varphi \r = \sum_{m=0}^n \int_0^1 \varphi^{(m)}(\tau)\, \dd \mu_m (\tau) \qquad \forall \varphi \in \mathcal{E}[0,1]. \]
\end{theorem*}
\begin{proof}
Let $u\in \mathcal{E}'[0,1]$ with $\mathrm{ord}(u)=n$. Define the map
\[ \Gamma: \mathcal{E}[0,1] \rightarrow \underbrace{C[0,1] \oplus \cdots \oplus C[0,1]}_{\textrm{$n+1$ copies}} : \varphi \mapsto \left( \varphi, \varphi', \ldots, \varphi^{(n)} \right), \]
where $C[0,1]$ denotes the Banach space of continuous functions on $[0,1]$ equipped with the supremum norm. Note that $\Gamma(\varphi)$ is uniquely determined by $\varphi$. Now consider the linear functional $\Lambda: \Gamma(\mathcal{E}[0,1])\rightarrow \mathbf{C}$ defined by
\[ \Lambda(\Gamma(\varphi)) = \l u, \varphi \r. \]
This map is bounded when $\Gamma(\mathcal{E}[0,1])$ is treated as a subspace of $C[0,1]\oplus \cdots \oplus C[0,1]$. Indeed, since $u \in \mathcal{E}'[0,1]$ has order $n$, we have
\[ \left| \Lambda(\Gamma(\varphi)) \right| = \left| \l u, \varphi \r\right| \leq C \sum_{m\leq n} \| \varphi^{(m)} \|_\infty \quad \forall \varphi \in \mathcal{E}[0,1]. \]
By the Hahn-Banach theorem, $\Lambda$ has an extension to all of $C[0,1] \oplus \cdots \oplus C[0,1]$, and by the Reisz-Markov theorem this extension must have the form
\[ \Lambda(\psi_0, \ldots, \psi_n ) = \sum_{m=0}^n \int_0^1 \psi_m(\tau)\, \dd \mu_m(\tau), \]
where $\{\mu_0, \ldots, \mu_n\}$ are complex Borel measures. In particular,
\[ \l u, \varphi \r = \Lambda(\Gamma(\varphi)) = \sum_{m=0}^n \int_0^1 \varphi^{(m)}(\tau)\, \dd \mu_m(\tau)  \quad \forall \varphi \in \mathcal{E}[0,1], \]
which is the representation stated in the theorem.
\end{proof}
A similar result holds if one considers the map
\[ \Gamma: \mathcal{E}[0,1] \rightarrow L^1[0,1] \oplus \cdots \oplus L^1[0,1]. \]
In this case one finds a collection $\{f_0, \ldots, f_n\}$ in $L^\infty[0,1]$, such that
\[ \l u, \varphi \r = \sum_{m=0}^n \int_0^1 f_m(\tau) \varphi^{(m)}(\tau)\, \dd \tau. \]
This is the form of the result used in \S 4.

\end{document}